\documentclass[12pt, reqno]{amsart}
\usepackage{amssymb}
\input xypic
\xyoption{all}

\newtheorem{thm}{Theorem}

\newtheorem{cor}[thm]{Corollary}
\newtheorem{lem}[thm]{Lemma}
\newtheorem{defi}[thm]{Definition}

\theoremstyle{remark}
\newtheorem{rem}[thm]{Remark}

\newtheorem{exa}[thm]{Example}

\newcommand{\calo}{{\mathcal O}}

\newcommand{\ideala}{{\mathfrak a}}
\newcommand{\idealb}{{\mathfrak b}}

\newcommand{\idealp}{{\mathfrak p}}

\newcommand{\idealq}{{\mathfrak q}}

\newcommand{\idealm}{{\mathfrak m}}

\newcommand{\spec}{{\rm Spec}}

\newcommand{\derk}{{\rm Der_k}}

\newcommand{\trd}{{\rm trdeg}}
\newcommand{\codim}{{\rm codim}}

\newcommand{\C}{{\mathbb C}}

\newcommand{\tor}{\xymatrix{\ar@{-->}[r]&}}

\begin{document}

\author{Rene Baltazar and Ivan Pan}

\thanks{Research of R. Baltazar  was partially supported by CAPES of Brazil}\thanks{Research of I. Pan was partially supported by ANII of Uruguay}

\title[On solutions for derivations]{On solutions for derivations of a Noetherian $k$-algebra and local simplicity}
\maketitle

\begin{abstract}
We introduce a general notion of solution for  a Noetherian differential $k$-algebra and study its relationship with simplicity, where $k$ is an  algebraically closed field; then we analyze conditions under which such solutions may exist and be unique, with special emphasis in the cases of $k$-algebras of finite type and formal series rings over $k$.  Using that notion we generalize a criterion for simplicity due to  Brumatti-Lequain-Levcovitz and give a geometric characterization of that; as an application we give a new proof of a classification theorem for local simplicity due to Hart and obtain a general result for simplicity of formal series rings over $k$.                  

\end{abstract}
\section{introduction}

In the classical theory of complex (or real) ordinary differential equations the Existence and Uniqueness Theorem asserts that if $D$ is an analytic vector field on $\C^n$ and $P\in \C^n$ is not a zero of $D$, then there exists an analytic map  $\gamma:\Delta\to \C^n$, where $\Delta\subset \C$ is an open disk containing $0$, such that $\gamma(0)=P=(p_1,\ldots,p_n)$ and $\gamma'(t)=D(\gamma(t)), t\in \Delta$.

Denote by $\calo_{n,P}$ the ring of germs of analytic functions at $P\in\C^n$ and think of $D$ as a derivation of that ring; hence $D=\sum_{i=1}^n f_i \partial_{z_i}$, where $f_i\in \calo_{n,P}$ and $z_1,\ldots, z_n$ are coordinates for $\C^n$. A solution as above is given by mean of $n$ functions $z_1(t),\ldots,z_n(t)$, which are  analytic in a neighborhood of $0$, and such that 
\[z'_i(t)=f_i(z_1(t),\ldots,z_n(t)), z_i(0)=p_i, i=1,\ldots,n.\] 
Since an element in $\calo_{n,P}$ may be represented as a power series in $z_1-p_1,\ldots,z_n-p_n$, with positive convergence radius, we obtain a unique $\C$-homomorphism $\varphi:\calo_{n,P}\to \calo_{1,0}$ which maps $z_i$ to $z_i(t)$; for an element  $g\in\calo_{n,P}$ we have $\varphi(g)(0)=0$ if and only if $g(P)=0$. Notice that, conversely, such a $\C$-homomorphism determines a unique solution. Motivated by this remark one may consider a $k$-derivation $D$ of an abstract $k$-algebra and say what a solution of $D$ means (see Definition \ref{defi1.1}). 

In \cite{Sei} one finds more or less implicitly, and besides a lot of important results, the essential material to study a more general notion of solution.  More explicitly,  \cite{Hart75} and \cite{BLL} consider solutions of suitable algebraic differential equations associated to a $k$-derivation to study simplicity of local differential rings. Furthermore, such a simplicity is characterized in \cite{Hart74} for an interesting class of local rings. The reader may also consult \cite[Thm 1.6.1]{Now} where there is a version of the Existence and Uniqueness Theorem for formal non autonomic differential equation systems.

In this work, which was inspired by \cite{BLL}, we introduce and study systematically a general notion of solution associated to a Noetherian differential $k$-algebra and its relationship with simplicity, for an algebraically closed field $k$. More precisely, in \S 2 we analyze conditions under which such solutions may exist and be unique with special emphasis in the cases of $k$-algebras of finite type and formal series rings over $k$ (Theorem \ref{thm1.2}). In \S 3  we first generalize the simplicity criterion given in \cite{BLL} and give a geometric characterization of that (Theorem \ref{thm2.1} and Corollary \ref{cor2.1_1}); next we give a new proof of the classification theorem for local simplicity \cite[Thm. 2]{Hart74} and obtain a general result for simplicity of formal series rings over $k$ (Corollary \ref{cor2.1_2} and Theorem \ref{thm2.2}).                  

\section{Existence and uniqueness Theorem} Let $k$ be an algebraically closed field of characteristic 0 and let $R$ be a $k$-algebra. We denote by $\derk(R)$ the $k$-vector space of $k$-derivations of $R$. If $R=k[[t]]$ we denote by $\partial_t$ the canonical derivation in $\derk(k[[t]])$.

Recall that an element $D\in\derk(R)$ extends in a unique form to a $k$-derivation in the total quotient ring of $R$ by the formula $D(a/s)=D(a)s^{-1}-as^{-2}D(s)$; then it also extends to any localization of $R$.   

\begin{defi}\label{defi1.1}
Let $D\in\derk(R)$ be a derivation and let $\idealp\in\spec(R)$ be a prime ideal; denote by $k(\idealp)$ the residue field of $\idealp$. A $k$-homomorphism $\varphi:R\to k(\idealp)[[t]]$ is said to be a solution of $D$ passing through $\idealp$ if $\varphi\circ D=\partial_t\circ\varphi$ and $\varphi^{-1}((t))=\idealp$.
When $\varphi(R)\not\subset k(\idealp)$ we say the solution is nontrivial. 
\end{defi}

A solution $\varphi:R\to K[[t]]$ as above, $K=k(\idealp)$, factorizes through the localization map $R\to R_\idealp$ to give a solution $\varphi_\idealp:R_\idealp\to K[[t]]$ passing through $\idealp R_\idealp$ (of $D$ thought as a derivation $R_\idealp\to R_\idealp$).

\begin{rem}
A solution $\varphi$ as in the definition above is trivial if and only if it induces a monomorphism $R/\ker\varphi\to k(\idealp)$; this signifies $\idealp=\ker\varphi$ is a maximal ideal.  
\end{rem}

If $D\in\derk(R)$, as pointed out in \cite{Sei} we may extend $D$ to an element in $\derk(R[[t]])$ and define the exponential $R$-automorphism $e^{tD}:R[[t]]\to R[[t]]$ given by
\[\alpha\mapsto \sum_{n=0}^\infty \frac{t^n}{n!}D^n(\alpha),\]
where $D^n=D^{n-1}\circ D$, for $n\geq 0$, and $D^0=Id$. Notice that $e^{tD}$ restricts to $R$ to give a $k$-homomorphism $R\to R[[t]]$.

In the sequel we think of $D$ as an element in $\derk(R)$ or $\derk(R[[t]])$ according to our convenience. 

Denote by $\epsilon_\idealp:R\to k(\idealp)=R_\idealp/\idealp R_\idealp$ the canonical map and let $\epsilon_\idealp\otimes 1:R\otimes _kk[[t]]=R[[t]]\to k(\idealp)\otimes_k k[[t]]=k(\idealp)[[t]]$ be its natural extension to power series. If $D\in\derk(R)$, then $\epsilon_\idealp\circ D$ is a $\epsilon_\idealp$-derivation of $R$.  We have:

\begin{lem}\label{lem1.1}
The map $(\epsilon_\idealp\otimes 1)\circ e^{tD}|_{R}:R\to k(\idealp)[[t]]$ defines a solution of $D$ passing through $\idealp$. Moreover, that solution is nontrivial if and only if $\epsilon_\idealp\circ D\not=0$.
\end{lem}

\begin{proof}
For $\alpha\in R$ we have
\begin{eqnarray*}
((\epsilon_\idealp\otimes 1)\circ e^{tD}\circ D)(\alpha)&=&\sum_{n=0}^\infty \frac{t^n}{n!}\epsilon_\idealp(D^{n+1}(\alpha))\\
&=&\partial_t\left(\sum_{n=-1}^\infty \frac{t^{n+1}}{(n+1)!}\epsilon_\idealp(D^{n+1}(\alpha))\right)\\
&=&\left(\partial_t\circ(\epsilon_\idealp\otimes 1)\circ e^{tD}\right)(\alpha).
\end{eqnarray*}
Clearly $(\epsilon_\idealp\otimes 1)\circ e^{tD}(\alpha)\in tk(\idealp)[[t]]$ if and only if $\epsilon_\idealp(\alpha)=0$, that is, if and only if $\alpha\in\idealp$. Moreover, if for an element $\alpha\in R$ we get $\epsilon_\idealp(D(\alpha))\neq 0$, then the linear term in the power series $\sum_{n=0}^\infty \frac{t^n}{n!}\epsilon_\idealp(D^{n}(\alpha))$ does not vanish; notice that $\epsilon_\idealp\circ D=0$ implies $\epsilon_\idealp\circ D^n=0, \ \forall n\geq 0$.  Putting all together we obtain the proof.
\end{proof}

\begin{rem}
All derivations admit at least a solution passing through a given prime ideal by Lemma \ref{lem1.1}. However, as shown in Example \ref{exa1.1}, for a given derivation all such solutions may be trivial. 
\end{rem}

If  $f:A\to B$ is a homomorphism of commutative rings, then $f^*:\spec(B)\to \spec(A)$ denotes the map $\idealp\mapsto \idealp^c:=f^{-1}(\idealp)$. It is continuous with respect to the related Zariski topologies. Recall that the correspondences $A\mapsto \spec(A)$ and $f\mapsto f^*$ induce an equivalence between the categories of $k$-algebras of finite type and affine varieties over $\spec(k)$, since $k$ is algebraically closed (Nullstelensatz).

\begin{defi}
Two $k$-homomorphisms $\varphi,\psi:R\to K[[t]]$ are said to be topologically equal if $\varphi^*=\psi^*$.  
\end{defi}

\begin{exa}
For $n\geq 1$ define the $k$-homomorphism $f_n:k[x]\to k[[t]]$ given by $x\mapsto t^n$. Then $f_n=f_m$ implies $n=m$, but  $f_n^*=f_m^*$ for all $m,n\geq 1$. 
\end{exa}

We denote by $k[[x_1,\ldots,x_n]]$ the power series $k$-algebra in $n$ indeterminates.
 
\begin{thm}\label{thm1.2}
Assume that $R$ is Noetherian and let $\idealp\in\spec(R)$. Let $D\in\derk(R)$ be a derivation. Then:

a) $D$ admits a solution passing through $\idealp$ which  is nontrivial if and only if $\epsilon_\idealp\circ D\neq 0$.

b) Two solutions of $D$ passing through $\idealp$ are topologically equal; in particular, if one of these is trivial, then the other one is too. 

c) If in addition $\idealm=\idealp$ is a maximal ideal and $R$ is either of finite type or a quotient algebra of the form $k[[x_1,\ldots,x_n]]/I$, where $I$ is an ideal of power series, then $D$ admits a unique solution passing through $\idealm$.
\end{thm}
\begin{proof}
The statement (a) follows from Lemma \ref{lem1.1}.

To prove (b) consider a solution $\varphi:R\to k(\idealp)[[t]]$ whose existence is assured by part (a). 

First of all we note that if $\varphi(\alpha)=0$, then $\varphi(D(\alpha))=\partial_t(\varphi(\alpha))=0$. Hence $D(\ker\varphi)\subset \ker\varphi$. Notice also that $\ker\varphi\subset \idealp$.

On the other hand, since $R$ is Noetherian there exists a unique prime ideal $\idealq\subset\idealp$ which is maximal among the ideals $\ideala\subset R$ satisfying $D(\ideala)\subset \ideala$ (see \cite[\S 3]{Sei}). From $D(\idealq)\subset \idealq$ we deduce $\partial_t(\varphi(\idealq))\subset \varphi(\idealq)$, which is contained in $tK[[t]]$. Since $\partial_t$ does not stabilize nontrivial ideals then  $\idealq\subset \ker\varphi$. We conclude that $\idealq$ is the kernel of any solution of $D$ passing through $\idealp$. Since, by definition, any solution contracts $tk[[t]]$ to $\idealp$, the statement (b) follows.

We prove (c) in the case $R=k[[x_1,\ldots,x_n]]/I$; the other case is analogous. By  Lemma \ref{lem1.1} we only need to prove the uniqueness statement. 

By assumption $D$ comes from an element $D_1\in \derk(k[[x_1,\ldots,x_n]])$ such that $D_1(I)\subset I$. We have $D_1=\sum_{i=1}^n f_i \partial /\partial x_i$, for some $f_i\in k[[x_1,\ldots,x_n]]$, $i=1,\ldots, n$. Let $\varphi: k[[x_1,\ldots,x_n]]\to k[[t]]$ be a solution of $D_1$ passing through a maximal ideal $M$ with $M/ I=\idealm$, and put $x_i(t):=\varphi(x_i)$, $i=1,\ldots,n$. Hence $M=(x_1-p_1,\ldots,x_n-p_n)$ where $p_i=x_i(0)$, $i=1,\ldots,n$.

Consider the $r$ truncation $k$-homomorphism $[\ \ ]_r:k[[t]]\to k[t]$, $r=0,1,2,\ldots$,  which to a power series $\sum_{i=0}^\infty a_i t^i$ associates  $\sum_{i=0}^r a_i t^i$. 

Recall that a power series $f\in k[[x_1,\ldots,x_n]]$ admits a Taylor development, around $p=(p_1,\ldots,p_n)$, as
\[f(x)=\sum_{j=0}^\infty \lambda^j(x-p),\]
where $x=(x_1,\ldots,x_n), \lambda^0(x-p)=f(p)$ and $\lambda^j\in k[x_1,\ldots,x_n]$ is a suitable homogeneous polynomial of degree $j$, for $j\geq 1$. Therefore 
\begin{equation}
f(x(t))\equiv \sum_{j=0}^r \lambda^j([x_1(t)-p_1]_r,\ldots,[x_n(t)-p_n]_r)\ \mbox{mod} (t^{r+1}).\label{eq1.1}
\end{equation}

Now, from  $\varphi\circ D_1=\partial_t\circ \varphi$ it follows 
\[f_i(x(t))=\partial_t x_i(t),\ i=1,\ldots n.\]
By applying  (\ref{eq1.1}) to $f_1,\ldots,f_n$ we deduce that the coefficient of the degree $r$ term of $\partial_t x_i(t)$ is determined by a finite number of coefficients of $f_i$ and the coefficients of $[x_1(t)-p_1]_{r},\ldots, [x_n(t)-p_n]_{r}$, for all $i=1,\ldots,n$. This proves that $\varphi$ is uniquely determined by the $f_j's$ and $p$. Since $D_1$ stabilizes $I$ we easily deduce that $\varphi$ factorizes through $k[[x_1,\ldots,x_n]]/I$ to give a (unique) solution of $D$ passing through $\idealm=M/I$. This completes the proof.

\end{proof}

Let $i$ be a nonnegative integer number. Following \cite{Hart74} we say $R$ is  $i$-singular at $\idealp\in\spec(R)$  if $\idealp R_\idealp$ can not be generated by $i+\dim R_\idealp$ elements. 

\begin{cor}
Suppose that  $R$ is either complete or a localization of a $k$-algebra of finite type; let $D\in\derk(R)$. All solutions of $D$ passing through a minimal $i$-singular prime is trivial. 
\end{cor}

\begin{proof}
By \cite[Thm. 1]{Hart74} we know $D$ stabilizes all minimal $i$-singular primes, $i\geq 0$. Theorem \ref{thm1.2}(a) implies there is a trivial solution passing through $\idealp$. The assertion follows from \ref{thm1.2}(b).   
\end{proof}

\begin{exa}\label{exa1.1}
Consider $R=k[x,y,z]$, $\idealp=(x,y)$ and the derivation $D=y\partial_x+xz\partial_y$. Hence $R/\idealp$ may be identified with $k[z]$ in such a way that $k(\idealp)=k(z)$ is the field of rational functions in one indeterminate.   

First of all note that the solution of $D$ (passing through $\idealp$) given by Lemma \ref{lem1.1} is defined by
\[x\mapsto 0, y\mapsto 0, z\mapsto z.\]
More generally, if $f\in k(z)$, a $k$ homomorphism given by 
\[x\mapsto 0, y\mapsto 0, z\mapsto f\]
 defines a solution of $D$ passing through $\idealp$. All these solutions are trivial (to compare with Theorem \ref{thm1.2}(b)).

On the other hand, a maximal ideal of $R$ is an ideal of the form  $\idealm_p=(x-a,y-b,z-c)$ for some $p=(a,b,c)\in k^3$.  The unique solution of $D$ passing through $\idealm_p$ is a $k$-homomorphism  $\varphi_p:R\to k[[t]]$ such that  $\varphi_p\circ D=\partial_t\circ \varphi_p$. In other words, we have 
\begin{equation}
y(t)=\partial_t x(t), x(t)z(t)=\partial_t y(t), 0=\partial_t z(t),\label{eq1.2}
\end{equation}
where $x(t):=\varphi_p(x), y(t):=\varphi_p(y), z(t):=\varphi_p(z)$. 

If $\idealm_p\supset\idealp$, then $p=(0,0,c)$ and $x(t)=0, y(t)=0, z(t)=c$ satisfy (\ref{eq1.2}). Otherwise, $p=(a,b,c)$ with $a\neq 0$ or $b\neq 0$. Since $D^{2m+1}(x)=yz^{m}, D^{2m}(x)=xz^{m}$, $D^{2m+1}(y)=xz^{m+1}$ and $D^{2m}(x)=yz^{m}$, we deduce that the unique solution passing through $\idealm_p\not\supset\idealp$ is given by $z(t)=c$ and
\[x(t)=\sum_{m=0}^\infty \left\{\frac{ac^m}{(2m)!}t^{2m}+\frac{bc^m}{(2m+1)!}t^{2m+1}\right\}, y(t)=\sum_{m=0}^\infty \left\{\frac{bc^m}{(2m)!}t^{2m}+\frac{ac^m}{(2m+1)!}t^{2m+1}\right\}.\]   
\end{exa}

\begin{exa}\label{exa1.2} Let $R$, $\idealp$ and $\idealm_p$ be as in the precedent example and let $D=\partial_x+\partial_y+\partial_z$. In this case the solutions passing through $\idealp$ are of the form
\[x\mapsto t, y\mapsto t, z\mapsto t+f\]
where $f\in k(z)$. Analogously $\varphi_p(x)=a+t, \varphi_p(y)=b+t, \varphi_p(z)=c+t$ is the unique solution passing through $\idealm_p$. 
\end{exa}

\section{On simplicity for local Noetherian $k$-algebras}

A (Noetherian) \emph{differential ring} is a pair $(R,D)$, where $R$ is a Noetherian $k$-algebra and $D\in\derk(R)$. Given two differential rings $(R_1,D_1), (R_2,D_2)$ a $k$-homomorphism $\psi:R_1\to R_2$ is said to be a \emph{morphism} of differential rings if $D_2\circ \psi=\psi\circ D_1$; we also write $\psi:(R_1,D_1)\to (R_2,D_2)$.  When $\psi$ is an isomorphism we say $(R_1,D_1)$ and  $(R_2,D_2)$ are isomorphic. 

Notice that a solution $\varphi:R\to k(\idealp)[[t]]$ of a derivation $D\in\derk(R)$, passing through a prime ideal $\idealp\in\spec(R)$, is nothing that a morphism $(R,D)\to (k(\idealp)[[t]],\partial_t)$, where one specifies the contraction of the maximal ideal. In particular, if $\psi:(R',D')\to (R,D)$ is a morphism, then $\varphi\circ\psi$ is a solution of $D'$ passing through $\psi^{-1}(\idealp)$.

\begin{defi}
A differential ring $(R,D)$ is said to be simple if the unique ideals stable under $D$ are $(0)$ and $R$. 
\end{defi}

Recall that a commutative ring is said to be \emph{reduced} if it has no nontrivial nilpotent elements.

Let $A$ be a commutative ring. For an ideal $\ideala\subset A$ we put $V(\ideala):=\{\idealq\in\spec(A); \idealq\supset\ideala\}$: it is the Zariski closed set associated to $\ideala$; note that for a prime ideal $\ideala=\idealq$,  $V(\idealq)$ is the Zariski closure in $\spec(A)$ of the single set $\{\idealq\}$. The dimension $\dim V(\ideala)$ of $V(\ideala)$ is $\dim A/\ideala$. If $\idealb\subset\ideala$ one obtains $V(\ideala)\subset V(\idealb)$ and $\dim V(\idealb)-\dim V(\ideala)$ is said to be the \emph{codimension} of $V(\ideala)$ in  $V(\idealb)$, denoted by $\codim(V(\ideala),V(\idealb))$.

If in addition $A$ is a $K$-algebra, where $K$ is a field,  we denote by $\trd_K A$ the transcendence degree of $A$ over $K$.

\begin{thm}\label{thm2.1}
Let $(R,D)$ be a differential ring, let $\idealp\in\spec(R)$ be a prime ideal in $R$. Assume that there exists a nontrivial solution $\varphi:R\to K[[t]]$ of $D$, $K=k(\idealp)$, passing through $\idealp$. The following assertions are equivalent:

(a) $(R_\idealp, D)$ is simple

(b) $\varphi_\idealp$ is one to one.

(c) $R_\idealp$ is reduced and the image of $\varphi_\idealp^*$ is dense in $\spec(R_\idealp)$.

(d) $\ker\varphi$ is the unique minimal prime contained in $\idealp$ and there is $u\in R\backslash \idealp$ such that $u\ker\varphi=0$.

If, in addition, $R$ is of finite type, then $(a), (b)$ and $(c)$ are equivalent to

(e) $R_\idealp$ is reduced and there is a unique irreducible component $X$ of $\spec(R)$ containing $V(\idealp)$ such that $\trd_K\varphi(R)=\codim(V(\idealp),X)$.
\end{thm}
\begin{proof}

Since $\varphi$ is nontrivial, then $\varphi_\idealp$ is too and hence $\ker\varphi_\idealp\subsetneq\idealp R_\idealp$.   As we noticed in the proof of Theorem \ref{thm1.2}, $\ker\varphi_\idealp$ is the biggest ideal in $R_\idealp$ which is stable under $D$. We deduce that $(a)$ and $(b)$ are equivalent. 

On the other hand the image of  $\varphi_\idealp^*$ is dense in $\spec(R_\idealp)$ if and only if  all element in $\ker\varphi_\idealp$ is nilpotent (see \cite[Chap. 1, Exercise. 21]{AM}), hence $(b)$ and $(c)$ are equivalent. 

Notice that the canonical map  $\lambda_\idealp:R\to R_\idealp$ induces a homeomorphism 
\begin{equation}\label{eq2.1}
\spec(R_\idealp)\simeq \{\idealq\in\spec(R); \idealq\subset\idealp\},
\end{equation} 
and consider the commutative diagram
\[\xymatrix{ &\spec(R_\idealp)\ar@{->}[d]^{\lambda_\idealp}\\
\spec(K[[t]])\ar@{->}[r]^{\varphi^*}\ar@{->}[ur]^{\varphi_\idealp^*}&\spec(R)}\]
Since the images of $\varphi^*$ and $\varphi_\idealp^*$ are dense in $V(\ker\varphi)$ and $V(\ker\varphi_\idealp)$, respectively, we deduce that the image of $\varphi_\idealp^*$ is dense in $\spec(R_\idealp)$ if and only if $V(\varphi_\idealp^*)=\spec(R_\idealp)$ if and only if  the right-hand side in (\ref{eq2.1}) is contained and dense in $V(\ker\varphi)$. This is equivalent to say that $\ker\varphi$ is the unique minimal prime contained in $\idealp$. Moreover, since the extension of $\ker\varphi$ in $R_\idealp$ is $\ker\varphi_\idealp$ we easily deduce that $(c)$ and $(d)$ are equivalent. This completes the first part of the proof.

Now suppose that $R$ is of finite type. Hence  $\dim R/ \ker\varphi=\trd_k\varphi(R)=\trd_K \varphi(R)+\dim R/\idealp$. In other words
\begin{equation}
\trd_K \varphi(R)=\codim(V(\idealp),V(\ker\varphi)).\label{eq2.2}
\end{equation}

If we assume that assertions $(a)$ to $(d)$ hold, then $X=V(\ker\varphi)$ is the  unique component of $\spec(R)$ containing $V(\idealp)$. From (\ref{eq2.2}) we know that component has the correct dimension, then $(e)$ holds. 

Conversely, assume that $(e)$ holds and notice that  $V(\ker\varphi)$ is an irreducible closed set in $\spec(R)$, which contains $V(\idealp)$. Then (\ref{eq2.2}) implies  $X=V(\ker\varphi)$, from which it follows that $\ker\varphi$ is the unique minimal prime contained in $\idealp$.  From (\ref{eq2.1}) we deduce, as before, that the image of $\varphi_\idealp^*$ is dense in $\spec(R_\idealp)$, i.e. all element in $\ker\varphi_\idealp$ is nilpotent. Thus $\varphi_\idealp$ is one to one which completes the proof.   
\end{proof}

\begin{exa}\label{exa2.1}\emph{Locally nilpotent derivations}. Let $(R,D)$ be a differential ring where $D$ is locally nilpotent, i.e. for each $a\in R$ there exists a positive integer $n$ such that $D^n(a)=0$. Suppose that $R$ is a (Noetherian) local ring with maximal ideal $\idealm$. Let $\ell$ be the minimum positive integer such that $D^\ell(\idealm)\subset \idealm$. Hence $D$ stabilizes the nonzero ideal generated by $D^{\ell-1}(\idealm)$. 

Suppose $\ell>1$ and consider a nontrivial solution of $D$ passing through $\idealm$.  From  Theorem \ref{thm1.2} and its proof we know that the kernel of such a solution coincides with the biggest ideal $\idealq\subset \idealm$ which is stable under $D$. Using the solution given by Lemma \ref{lem1.1} we deduce
\[\idealq=\{a\in R; D^{i}(a)\in\idealm, i=0,\ldots,\ell-1\}.\]

By using theorem  \ref{thm2.1} we conclude that  $(R,D)$ is simple if and only if one of the following equivalent conditions holds:
\begin{itemize}
\item there exist $a_1,\ldots,a_s\in R$ and $x_1,\ldots,x_s\in\idealm$ such that $\sum_{i=1}^s a_iD^{\ell-1}(x_i)=1$.

\item $a\in\idealm, D(a)\in\idealm,\ldots, D^{\ell-1}(a) \in\idealm$ imply $a=0$.
\end{itemize}
\end{exa}

The following corollary generalizes the results in \cite[\S 1]{BLL} (see Remark \ref{rem_BLL}).

\begin{cor}\label{cor2.1_1}
Let $R$ be a $k$-algebra of finite type without zero divisors and let $\idealp\in\spec(R)$.  If  $D\in\derk(R)$ is a derivation such that $D(\idealp)\not\subset \idealp$, then there is at least a nontrivial solution $\varphi:R\to K[[t]]$ passing through $\idealp$ and  the following assertions are equivalent:

(a) $(R_\idealp,D)$ is simple.

(b) $\varphi_\idealp$ is one to one.

(c) the image of $\varphi_\idealp^*$ is dense in $\spec(R_\idealp)$.

(d) $\trd_K\varphi(R)=\codim(V(\idealp),\spec(R))$.
\end{cor}

\begin{proof}
Since $D(\idealp)\not\subset \idealp$ implies $\epsilon_{\idealp}\circ D\neq 0$ the existence of a nontrivial solution passing through $\idealp$ is assured by Theorem \ref{thm1.2}. Taking into account that in the present case $R$, and then $R_\idealp$, have no zero divisors,  the corollary follows readily from Theorem \ref{thm2.1}.
\end{proof}

\begin{rem}\label{rem_BLL}
If $R=k[x,y_1,\ldots,y_r]$, $\idealp$ is the maximal ideal $\idealm=(x-\alpha,y_1-\beta_1,\ldots,y_r-\beta_r)$, $(\alpha,\beta_1,\ldots,\beta_r)\in k^{r+1}$, and $D\in\derk(R)$ is a derivation, then $D(\idealm)\not\subset\idealm$ if and only if $D=g\partial_x+\sum_{i=1}^r f_i\partial_{y_i}$, where $g,f_1,\ldots,f_r\in k[x,y_1,\ldots,y_r]$ are polynomials not all of them belonging to $\idealm$. From the unique nontrivial solution for $(R,D)$, passing through $\idealm$, we obtain a solution $\varphi_\idealm$ of $(R_\idealm,D)$. If we put $x(t):=\varphi_\idealm(x), y_i(t):=\varphi_\idealm(y_i), i=1,\ldots,r$, then we may read Corollary \ref{cor2.1_1} as saying $(R_\idealm, D)$ is simple if and only if $x(t),y_1(t),\ldots,y_r(t)$ are transcendent over $k$.    
\end{rem}

The second part of the next result is essentially \cite[Thm. 2]{Hart74}.

\begin{cor}\label{cor2.1_2}
Let $S$ be a Noetherian local $k$-algebra with maximal ideal $\idealm$ and let $D\in\derk(S)$. Then $(S,D)$ is simple if and only if there is a one to one  solution passing through $\idealm$. In particular, $(S, D)$ is isomorphic to a differential ring $(S_0, D_0)$, where $S_0$ is a $k$-subalgebra of $K[[t]]$, $K=k(\idealm)$, which is stable under $\partial_t$ and $D_0:=\partial_t|_{S_0}$.   
\end{cor}
\begin{proof}
Note that simplicity implies $D(\idealm)\not\subset\idealm$, from which we know $D$ admits nontrivial solutions (Theorem \ref{thm1.2}(a)).  By Theorem \ref{thm2.1} we deduce $S$ is $D$-simple if and only if there exists a one to one solution, say  $\varphi:S\to K[[t]]$. For the rest of the proof we take $S_0=\varphi(S)$, and the rest of the assertion is essentially trivial. 
\end{proof}

For the completion of a  $k$-algebra of finite type, with respect to a maximal ideal, simplicity is quite rare. In fact, suppose $R=k[[x_1,\ldots,x_n]]$ and let $\varphi:R\to k[[t]]$ be the unique solution passing through $\idealm=(x_1,\ldots,x_n)$, associated to a derivation $D$. Corollary \ref{cor2.1_2} implies $\ker\varphi\neq 0$ unless $n=1$ and $D=f\partial_{x_1}$ for some $f\in k[[x_1]]$ with $f(0)\neq 0$. More generally, we have:

\begin{thm}\label{thm2.2}
Let $R$ be the quotient of $k[[x_1,\ldots,x_n]]$ by an ideal $I$. If $D\in\derk(R)$, then $(R,D)$ is simple if and only if $D$ lifts to a derivation $\widehat{D}\in\derk(k[[x_1,\ldots,x_n]])$ which admits  a unique solution $\widehat{\varphi}:k[[x_1,\ldots,x_n]]\to k[[t]]$, passing through the maximal ideal $(x_1,\ldots,x_n)$, such that $\ker\widehat{\varphi}=I$.
\end{thm}  
\begin{proof}
Note that each element in $\derk(R)$ comes from an element in $\derk(k[[x_1,\ldots,x_n]])$ which stabilizes $I$ and that under this correspondence we obtain compatible solutions. Recalling that two solutions passing through the maximal ideal have the same kernel (Theorem \ref{thm1.2}b) the result follows from Corollary \ref{cor2.1_2}.
\end{proof}

\end{document}